\documentclass{amsart}
\usepackage{amsmath,amssymb,graphicx,latexsym}
\usepackage{hyperref}
\usepackage[T1]{fontenc}
\usepackage{textcomp}
\usepackage[utf8]{inputenc}
\usepackage{mathtools} 
\usepackage{amsfonts} 
\usepackage{times, url}
\usepackage{siunitx}
\usepackage{float}
\usepackage{listings}
\usepackage{subcaption}
\usepackage{nccmath}
\usepackage[english]{babel}
\newtheorem{theorem}{Theorem}
\newtheorem{lemma}[theorem]{Lemma}

\newtheorem{conjecture}[theorem]{Conjecture}

\theoremstyle{definition}

\newtheorem{remark}[theorem]{Remark}
\newtheorem*{definition}{Definition}
\newtheorem*{notation*}{Notation}

\begin{document}
\title{On some nested floor functions and their jump discontinuities}
\author{Luca Onnis}
\date{March 2022}
\lstloadlanguages{[5.2]Mathematica}

\maketitle
\begin{abstract}
    This paper investigates some particular limits involving nested floor functions. We'll prove some cases and then we'll show a more general result. Then we'll count the discontinuity points of those functions, and we'll prove a method to find them all. Surprisingly the set of the jump discontinuities of $f_n$ is a subset of the set of the jump discontinuities of $f_{n+1}$, $\forall n\in\mathbb{Z^{+}}$ where:
    \[
    f_n(x)=\underbrace{\Biggl\lfloor x\Bigl\lfloor x \lfloor\dots\rfloor\Bigr\rfloor\Biggr\rfloor}_{\text{$n$ times}}
    \]
    Furthermore we'll give some generalizations of the result and lots of considerations; for example we'll prove that the cardinality of the set of the discontinuities of $f_n$ in a given limited interval approaches infinity as $n\to\infty$.
\end{abstract}
\tableofcontents
\section{Introduction}
\begin{definition}
In mathematics and computer science, the floor function is the function that takes as input a real number $x$, and gives as output the greatest integer less than or equal to $x$, denoted $\lfloor x \rfloor$.
\end{definition}
\subsection{First problem}
Let $k\in\mathbb{Z^{+}}\setminus\{1\}$ and $n\in\mathbb{Z^{+}}$; we want to compute the following limit:
\[
\lim_{x\to k^{-}}\underbrace{\Biggl\lfloor x\Bigl\lfloor x \lfloor\dots\rfloor\Bigr\rfloor\Biggr\rfloor}_{\text{$n$ times}}
\]
by varying $k,n$.
\subsection{Example} \label{ex1}
Let $k=3$ and $n=2$, we want to compute:
\[
\lim_{x\to 3^{-}}\Bigl\lfloor x \lfloor x\rfloor\Bigr\rfloor
\]
The answer of this case is 5, and it's possible to prove it using the definition of limit: \\
$\forall\varepsilon>0\exists\delta>0$ such that:
\[
\Bigl|\Bigl\lfloor x \lfloor x\rfloor\Bigr\rfloor-5\Bigr|<\varepsilon \mbox{ $\forall x\in(3-\delta,3)$}
\]
Using $\delta=\frac{1}{2}$, we have $x\in(\frac{5}{2},3)$. So:
\[
\frac{5}{2}<x<3 \Rightarrow \lfloor x \rfloor = 2
\]
And then, by replacying it in our definition we have:
\[
\Bigl|\lfloor 2x \rfloor - 5\Bigr|<\varepsilon
\]
But furthermore since $\frac{5}{2}<x<3$ we know that:
\[
\frac{5}{2}<x<3 \Rightarrow 5<2x<6 \Rightarrow \lfloor 2x \rfloor = 5
\]
and finally:
\[
|5-5|=0<\varepsilon \mbox{ $\forall\varepsilon\in\mathbb{R^{+}}$}
\]
So:
\[
\lim_{x\to 3^{-}}\Bigl\lfloor x \lfloor x\rfloor\Bigr\rfloor = 5
\]
\section{First generalization}
In this section we will generalize the example \ref{ex1}.
\begin{theorem} \label{t1}
Let $k\in\mathbb{Z^{+}}\setminus\{1\}$ and $n\in\mathbb{Z^{+}}$:
\[
\lim_{x\to k^{-}}\underbrace{\Biggl\lfloor x\Bigl\lfloor x \lfloor\dots\rfloor\Bigr\rfloor\Biggr\rfloor}_{\text{$n$ times}}=\frac{1}{k-1}\Bigl[(k-2)\cdot k^{n}+1\Bigr]
\]
\end{theorem}
\begin{proof}
It is sufficient to prove that $\forall k\in\mathbb{Z^{+}}\setminus\{1\}$ and $\forall n\in\mathbb{Z^{+}}$ $\exists\delta_n=\frac{k-1}{(k-2)k^{n-1}+1}$ such that
\[
\underbrace{\Biggl\lfloor x\Bigl\lfloor x \lfloor\dots\rfloor\Bigr\rfloor\Biggr\rfloor}_{\text{$n$ times}}=\frac{1}{k-1}\Bigl[(k-2)\cdot k^{n}+1\Bigr] \mbox{ $\forall x\in (k-\delta_n,k)$}
\]
The base case is when $n=1$. We have that:
\[
\lfloor x \rfloor = k-1 \mbox{ $\forall x\in(k-\frac{1}{k-1},k)$}
\]
and this identity is true because $k$ is a positive integer. \\
Suppose that this identity holds for $n$ and we'll prove it right for $n+1$. \\
Since $k^{n}>k^{n-1} \forall n\in\mathbb{Z^{+}}\forall k\in\mathbb{Z^{+}}\setminus\{1\}$, consider:
\[
\delta_{n+1}=\frac{k-1}{(k-2)k^{n}+1}<\frac{k-1}{(k-2)k^{n-1}+1}=\delta_n
\]
we want to show that:
\[
\underbrace{\Biggl\lfloor x\Bigl\lfloor x \lfloor\dots\rfloor\Bigr\rfloor\Biggr\rfloor}_{\text{$n+1$ times}}=\frac{1}{k-1}\Bigl[(k-2)\cdot k^{n+1}+1\Bigr] \mbox{ $\forall x\in (k-\delta_{n+1},k)$}
\]
but since $\delta_{n+1}<\delta_n$ we have that:
\[
\underbrace{\Biggl\lfloor x\Bigl\lfloor x \lfloor\dots\rfloor\Bigr\rfloor\Biggr\rfloor}_{\text{$n+1$ times}}=\Biggl\lfloor\frac{1}{k-1}\Bigl[(k-2)\cdot k^{n}+1\Bigr]x\Biggr\rfloor \mbox{ $\forall x\in I=(k-\delta_{n+1},k)\subset(k-\delta_n,k)$}
\]
Furthermore:
\[
\Biggl\lfloor\frac{1}{k-1}\Bigl[(k-2)\cdot k^{n}+1\Bigr]x\Biggr\rfloor=\frac{1}{k-1}\Bigl[(k-2)\cdot k^{n+1}+1\Bigr] \mbox{ $\forall x\in I=(k-\delta_{n+1},k)$}
\]
This is true in fact:
\[
I=(k-\delta_{n+1},k)=\Biggl(\frac{k^{n+2}-2k^{n+1}+1}{k^{n+1}-2k^{n}+1},k\Biggr)
\]
So:
\[
\frac{k^{n+2}-2k^{n+1}+1}{k^{n+1}-2k^{n}+1}<x<k \Rightarrow k^{n+2}-2k^{n+1}+1<(k^{n+1}-2k^{n}+1)x<k^{n+2}-2k^{n+1}+k
\]
\[
(k-2)k^{n+1}+1<[(k-2)k^{n}+1]x<(k-2)k^{n+1}+k
\]
And dividing all by $\frac{1}{k-1}$, we'll get:
\[
\underbrace{\frac{(k-2)k^{n+1}}{k-1}+\frac{1}{k-1}}_{\text{$n_1\in\mathbb{N}$}}<\frac{1}{k-1}\Bigl[(k-2)k^{n}+1\Bigr]x<\underbrace{\frac{(k-2)k^{n+1}}{k-1}+\frac{k}{k-1}}_{\text{$n_2\in\mathbb{N}$}}
\]
Call $n_1$ the left hand side of the inequality and $n_2$ the right hand side. We'll prove that they are in fact natural numbers in the next section. Note that:
\[
0<n_2-n_1=1
\]
So the number $\frac{1}{k-1}\Bigl[(k-2)k^{n}+1\Bigr]x$ is strictly between $n_1,n_2$ ,which are positive integers whose distance between each others is equal to $1$. So the floor of 
\[
\frac{1}{k-1}\Bigl[(k-2)k^{n}+1\Bigr]x
\]
must be equal to $n_1$ (which is the nearest integer less than or equal to \\
$\frac{1}{k-1}\Bigl[(k-2)k^{n}+1\Bigr]x$). \\
Let $\delta_n=\frac{k-1}{(k-2)k^{n-1}+1}$. We want to show that:
\[
\Biggl|\underbrace{\Biggl\lfloor x\Bigl\lfloor x \lfloor\dots\rfloor\Bigr\rfloor\Biggr\rfloor}_{\text{$n$ times}}-\frac{1}{k-1}\Bigl[(k-2)\cdot k^{n}+1\Bigr]\Biggr|<\varepsilon \mbox{ $\forall x\in(k-\delta_{n},k)$}
\]
But now this is obvious because:
\[
\underbrace{\Biggl\lfloor x\Bigl\lfloor x \lfloor\dots\rfloor\Bigr\rfloor\Biggr\rfloor}_{\text{$n$ times}}=\frac{1}{k-1}\Bigl[(k-2)\cdot k^{n}+1\Bigr] \mbox{ $\forall x\in (k-\delta_n,k)$}
\]
And finally $\forall\varepsilon>0$:
\[
\Biggl|\underbrace{\Biggl\lfloor x\Bigl\lfloor x \lfloor\dots\rfloor\Bigr\rfloor\Biggr\rfloor}_{\text{$n$ times}}-\frac{1}{k-1}\Bigl[(k-2)\cdot k^{n}+1\Bigr]\Biggr|=0<\varepsilon \mbox{ $\forall x\in(k-\delta_{n},k)$}
\]
And the thesis follows from the limit definition.
\end{proof}
\begin{lemma}
$n_1$ and $n_2$ defined in Theorem \ref{t1} are positive integers. \\
We want to prove that:
\[
(k-2)k^{n+1}+1\equiv 0 \mod (k-1)
\]
\[
k^{n+2}-2k^{n+1}+1\equiv 0\mod (k-1)
\]
$\forall k\in\mathbb{Z^{+}}\setminus\{1\}$ and $n\in\mathbb{Z^{+}}$. 
\end{lemma}
\begin{proof}
Note that:
\[
k^{n+2}-2k^{n+1}+1=(k-1)(k^{n+1}-k^{n}-k^{n-1}-\dots-k-1)
\]
So the numerator of $n_1$ is a multiple of $(k-1)$. Since the numerator of $n_2$ is equal to the numerator of $n_2$ plus $k-1$ we conclude that also $n_2$ is an integer (because its numerator is again a multiple of $k-1$).
\end{proof}
\section{Jump discontinuities}
Let $f(x)$ be the function:
\[
f_{n}(x) = \underbrace{\Biggl\lfloor x\Bigl\lfloor x \lfloor\dots\rfloor\Bigr\rfloor\Biggr\rfloor}_{\text{$n$ times}}
\]
Note that:
\[
\lim_{x\to k^{-}} f_{n}(x) = \frac{1}{k-1}\Bigl[(k-2)\cdot k^{n}+1\Bigr] \mbox{ while } \lim_{x\to k^{+}} f_{n}(x) = k^{n} \mbox{ $\forall k\in\mathbb{Z^{+}}$}
\]
From these relations we know that $x = k$ is a a point of discontinuity of the first kind $\forall k\in\mathbb{Z^{+}}$ with jump's length equal to:
\[
|J(k,f_n)|=k^{n} - \frac{1}{k-1}\Bigl[(k-2)\cdot k^{n}+1\Bigr] = \frac{k^{n}-1}{k-1}
\]
But the function $f_{n}(x)$ has more jump discontinuities than these. For example:
\[
f_2(x)=\Bigl\lfloor x \lfloor x \rfloor\Bigr\rfloor
\]
has a jump discontinuity in $x=\frac{10}{3}$, in fact:
\[
\lim_{x\to \frac{10}{3}^{-}} \Bigl\lfloor x \lfloor x \rfloor\Bigr\rfloor = 9 \mbox{ while } \lim_{x\to \frac{10}{3}^{+}} \Bigl\lfloor x \lfloor x \rfloor\Bigr\rfloor = 10
\]
As you can see from these graphs:
\begin{figure}[h!]
  \centering
\begin{subfigure}[b]{0.4\linewidth}
    \includegraphics[width=\linewidth]{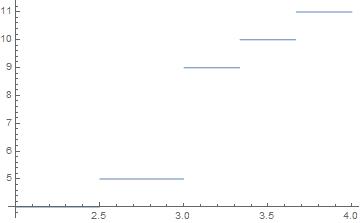}
    \caption{Graph of $f_2(x)$ in $I=[2,4]$}
  \end{subfigure}
\begin{subfigure}[b]{0.4\linewidth}
    \includegraphics[width=\linewidth]{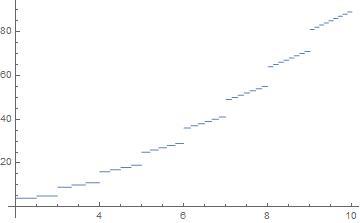}
    \caption{Graph of $f_2(x)$ in $I=[2,10]$}
  \end{subfigure}
  \caption{Graph of $f_2(x)$}
  \label{fig:coffee}
\end{figure}
\begin{remark}
Let $f_1(x)=\lfloor x \rfloor$, $P(a,b,f_1)$ the set of discontinuity points of $f_1(x)$ in the interval $[a,b)$ and $P(f_1)$ the set of all discontinuity points over the domain $x\geq 1$. Then:
\[
P(f_1) = \mathbb{Z^{+}}
\]
In fact every $x = k$ (where $k\in\mathbb{Z^{+}}$) is a jump discontinuity for $f_1$, where:
\[
J(k,f_1)=1 \mbox{ $\forall k\in P(f_1)$}
\]
\end{remark}
\begin{theorem} \label{t4}
Let $f_2(x)$ be defined as before, $P(a,b,f_2)$ the set of discontinuity points of $f_2(x)$ in the open interval $[a,b)$ and $P(f_2)$ the set of all discontinuity points over the domain $x\geq 1$. Then:
\[
P(f_2) = \bigcup_{k=1}^{+\infty}P(k,k+1,f_2) = \bigcup_{k=1}^{+\infty}\Bigl\{k,k+\frac{1}{k},k+\frac{2}{k},\dots,k+\frac{k-1}{k}\Bigr\}
\]
Or:
\[
P(f_2) = \bigcup_{k=1}^{+\infty}\bigcup_{r=0}^{k-1}\Bigl\{k+\frac{r}{k}\Bigr\}
\]
where:
\[
|J(k,f_2)| = \frac{k^{2}-1}{k-1}=k+1 \mbox{ $\forall k\in\mathbb{Z^{+}}$}
\]
and:
\[
\Bigl|J\Bigl(k+\frac{r}{k}, f_2\Bigr)\Bigr| = 1 \mbox{ $\forall k\in\mathbb{Z^{+}}$ where $r\in\{1,2,\dots,k-1\}$}
\]
\end{theorem}
\subsubsection{Examples}
For example, consider the function $f_2(x)$ and the interval $[4,5)$. Assuming true Theorem 2 we know that:
\[
P(4,5,f_2)=\Bigl\{4,\frac{17}{4},\frac{9}{2},\frac{19}{4}\Bigr\}
\]
In fact:
\[
\lim_{x\to 4^{-}}f_2(x)=11 \mbox{ $\wedge$ } \lim_{x\to 4^{+}}f_2(x)=16
\]
and:
\[
|J(4,f_2)| = 4+1 = 5 \mbox{ the jump is in fact: $16-11$}
\]
While:
\[
\lim_{x\to \frac{17}{4}^{-}}f_2(x)=16 \mbox{ $\wedge$ } \lim_{x\to \frac{17}{4}^{+}}f_2(x)=17
\]
\[
\lim_{x\to \frac{9}{2}^{-}}f_2(x)=17 \mbox{ $\wedge$ } \lim_{x\to \frac{9}{2}^{+}}f_2(x)=18
\]
\[
\lim_{x\to \frac{19}{4}^{-}}f_2(x)=18 \mbox{ $\wedge$ } \lim_{x\to \frac{19}{4}^{+}}f_2(x)=19
\]
So:
\[
\Bigl|J\Bigl(\frac{17}{4},f_2\Bigr)\Bigr| = \Bigl|J\Bigl(\frac{9}{2},f_2\Bigr)\Bigr| = \Bigl|J\Bigl(\frac{19}{4},f_2\Bigr)\Bigr| = 1
\]
As you can see from this image:
\begin{figure}[h]
    \centering
    \includegraphics[width=3cm]{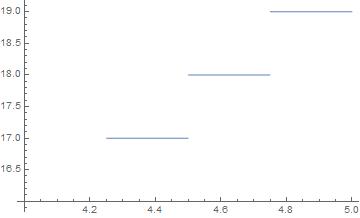}
    \caption{Graph of $f_2(x)$ in $I=[4,5)$}
    \label{fig:4.2}
\end{figure}
\begin{proof}
We've already proved that $ x = k$ is a discontinuity point $\forall k\in\mathbb{Z^{+}}$, we should prove that:
\[
\Bigl\{k+\frac{1}{k},\dots,k+\frac{k-1}{k}\Bigr\}_{k=1}^{k=+\infty} \mbox{ are the only others discontinuity points of $f_2$}
\]
So we're proving that:
\[ 
\bigcup_{k=1}^{+\infty}\Bigl\{k,k+\frac{1}{k},k+\frac{2}{k},\dots,k+\frac{k-1}{k}\Bigr\} \subseteq P(f_2) \wedge P(f_2) \subseteq \bigcup_{k=1}^{+\infty}\Bigl\{k,k+\frac{1}{k},k+\frac{2}{k},\dots,k+\frac{k-1}{k}\Bigr\}
\]
Let:
\[
\lim_{x\to (k+\frac{r}{k})^{-}} f_2(x) = L(k,r)^{-} \mbox{ and } \lim_{x\to (k+\frac{r}{k})^{+}} f_2(x) = L(k,r)^{+}
\]
Then we'll prove that $\forall r\in\{1,2,\dots,k-1\}$:
\[
L(k,r)^{-}=k^{2}+r-1 \wedge L(k,r)^{+}=k^{2}+r
\]
Using $\delta = \frac{1}{k}$ we have that:
\[
\Biggl|\Bigl\lfloor x \lfloor x \rfloor\Bigr\rfloor - (k^{2}+r-1)\Biggr|<\varepsilon \mbox{ $\forall x\in (k+\frac{r-1}{k},k+\frac{r}{k})$}
\]
We know that:
\[
k+\frac{r-1}{k}<x<k+\frac{r}{k} \Rightarrow \lfloor x \rfloor = k \mbox{ $\forall x \in (k+\frac{r-1}{k},k+\frac{r}{k})$}
\]
And:
\[
k+\frac{r-1}{k}<x<k+\frac{r}{k} \Rightarrow k^{2}+r-1<kx<k^{2}+r \Rightarrow \lfloor kx \rfloor = k^{2}+r-1
\]
So:
\[
\Biggl|\Bigl\lfloor x \lfloor x \rfloor\Bigr\rfloor - (k^{2}+r-1)\Biggr| = \Biggl|\Bigl\lfloor kx \Bigr\rfloor - (k^{2}+r-1)\Biggr| = \Biggl|k^{2}+r-1 - (k^{2}+r-1)\Biggr|=0<\varepsilon
\]
Similarly, using $\delta=\frac{1}{k}$ again it's possible to prove with the same technique that:
\[
\Biggl|\Bigl\lfloor x \lfloor x \rfloor\Bigr\rfloor - (k^{2}+r)\Biggr|<\varepsilon \mbox{ $\forall x\in (k+\frac{r}{k},k+\frac{r+1}{k})$}
\]
We prove the first inequality. In order to prove the second inequality it's sufficient to note that $f_2$ is a monotone increasing function over its domain and that $f_2(x)\in\mathbb{N} , \forall x\geq 1$. Furthermore:
\[
\Bigl|J\Bigl(k+\frac{r}{k},f_2\Bigr)\Bigr|=1 \mbox{ $\forall k\in\mathbb{Z^{+}}$ where $r\in\{1,2,\dots,k-1\}$}
\]
Since:
\[
L(k,r)^{-}=k^{2}+r-1 \wedge L(k.r)^{+}=k^{2}+r
\]
So we have that:
\[
k^{2}+r-1\leq f_2(x) < k^{2}+r \mbox{ $\forall x \in\Bigl[k+\frac{r-1}{k},k+\frac{r}{k}\Bigr)$}
\]
But $f_2(x)\in\mathbb{N}$, so the function in that interval is constant, and is equal to:
\[
f_2(x)=k^{2}+r-1 \mbox{ $\forall x \in\Bigl[k+\frac{r-1}{k},k+\frac{r}{k}\Bigr)$}
\]
Finally we prove that every discontinuity point of $f_2$ are elements of the following set:
\[
P(f_2) = \bigcup_{k=1}^{+\infty}\Bigl\{k,k+\frac{1}{k},k+\frac{2}{k},\dots,k+\frac{k-1}{k}\Bigr\}
\]
Or:
\[
P(f_2) = \bigcup_{k=1}^{+\infty}\bigcup_{r=0}^{k-1}\Bigl\{k+\frac{r}{k}\Bigr\}
\]
\end{proof}
\subsection{First considerations}
The set of discontinuity points of $f_2(x)$ is a countable set. In fact it's the countable union of countable sets. \cite{1} So:
\[
|P(f_2)| = |\mathbb{N}| = \aleph_0
\]
Let $h\in\mathbb{Z^{+}}\setminus\{1\}$ , the cardinality of the finite set defined as:
\[
|P(1,h,f_2)|= \Bigl|\bigcup_{k=1}^{h-1}\Bigl\{k,k+\frac{1}{k},k+\frac{2}{k},\dots,k+\frac{k-1}{k}\Bigr\}\Bigl|
\]
is equal to:
\[
|P(1,h,f_2)|=1+2+3+\dots+h-1 =\frac{h(h-1)}{2}
\]
\subsection{Generalizations}
If we consider the functions $f_3,f_4,f_5,\dots,f_n$ it's easy to see that there are more and more discontinuity points as $n$ increases. Back to the $f_2$ case, it's possible to construct a partition of a generic interval $I=[a,b)$ , made of the discontinuity points of $f_2$ in that interval. For example, let $I=[3,4)$, then:
\[
P(3,4,f_2)=\Bigl\{3,3+\frac{1}{3},3+\frac{2}{3}\Bigr\}
\]
while we'll prove that:
\[
P(3,4,f_3)=\Bigl\{3,3+\frac{1}{9},3+\frac{2}{9},3+\frac{1}{3},3+\frac{2}{5},3+\frac{1}{2},3+\frac{3}{5},3+\frac{2}{3},3+\frac{8}{11},3+\frac{9}{11},3+\frac{10}{11}\Bigr\}
\]
Note that:
\[
P(3,4,f_2)\subset P(3,4,f_3) \mbox{ $\wedge$ } |P(3,4,f_2)|<|P(3,4,f_3)|
\]
\begin{theorem} \label{t5}
Let $P(f_n)$ denotes the set of the discontinuity points of the function $f_n$. Then:
\[
P(f_1)\subset P(f_2) \subset\dots\subset P(f_n) \mbox{ $\forall n\in\mathbb{Z^{+}}$}
\]
\end{theorem}
\begin{proof}
We'll prove this result by induction on $n$. \\
The base case ($n=2$) has been already proved before. So we know that \\$P(f_1)\subset P(f_2)$. \\
Suppose that $P(f_1)\subset P(f_2)\subset \dots \subset P(f_{n-1})$. We'll prove that:
\[
P(f_{n-1})\subset P(f_n)
\]
Let $d$ be an element of $P_{n-1}$; consider the following limits:
\[
\lim_{x\to d^{-}} f_{n-1}(x) = L_d^{-} \mbox{ $\wedge$ } \lim_{x\to d^{+}} f_{n-1}(x) = L_d^{+}
\]
From the induction hyphothesis we know that:
\[
L_d^{-}\not=L_d^{+}
\]
So $\exists\delta^{-},\delta^{+}>0$ such that:
\[
\Bigl|f_{n-1}(x)-L_d^{-}\Bigr|=0 \mbox{ $\forall x\in (d-\delta^{-},d)$}
\]
\[
\Bigl|f_{n-1}(x)-L_d^{+}\Bigr|=0 \mbox{ $\forall x\in (d,d+\delta^{+})$}
\]
But from the definition of $f_n$ we have that:
\[
f_n(x)=\lfloor x\cdot f_{n-1}(x) \rfloor
\]
So:
\[
\lim_{x\to d^{-}}f_n(x)=\lim_{x\to d^{-}}\lfloor x\cdot f_{n-1}(x)\rfloor
\]
\[
\lim_{x\to d^{+}}f_n(x)=\lim_{x\to d^{+}}\lfloor x\cdot f_{n-1}(x)\rfloor
\]
But $f_{n-1}=L_d^{-},\forall x\in (d-\delta^{-},d)$ and  $f_{n-1}=L_d^{+},\forall x\in (d,d+\delta^{+})$, so substituing in we will obtein:
\[
\lim_{x\to d^{-}} f_{n}(x) = \lim_{x\to d^{-}}\lfloor x\cdot L_d^{-}\rfloor 
\]
\[
\lim_{x\to d^{+}}f_{n}(x)= \lim_{x\to d^{+}}\lfloor x\cdot L_d^{+}\rfloor
\]
This last step is motivated by using the limit definition. \\
But we know that $L_d^{-},L_d^{+}\in\mathbb{Z^{+}}$ , and since $f_n$ is monotone increasing and $L_d^{-}\not= L_d^{+}$ we can conclude that:
\[
L_d^{+} > L_d^{-} \implies L_d^{+} = L_d^{-} + k \mbox{ for some $k\in\mathbb{Z^{+}}$}
\]
So:
\[
\lim_{x\to d^{-}} f_{n}(x) = \lim_{x\to d^{-}}\lfloor x\cdot L_d^{-}\rfloor 
\]
\[
\lim_{x\to d^{+}}f_{n}(x)= \lim_{x\to d^{+}}\lfloor x\cdot L_d^{-} + kx\rfloor
\]
But using $\delta=\min\{\delta^{-},\delta^{+}\}$ we have:
\[
d<x<d+\delta \implies kd<kx<kd+k\delta
\]
And finally we have that $kx$ in this interval is bigger than $kd$ (which is a positive rational number bigger than or equal to 1). So:
\[
\lim_{x\to d^{+}}f_{n}(x)= \lim_{x\to d^{+}}\lfloor x\cdot L_d^{-} + kx\rfloor = \lim_{x\to d^{+}}\lfloor x\cdot L_d^{-} + \underbrace{kx-\lfloor kx \rfloor}_{\text{$\geq 0$}} \rfloor + \lfloor kx \rfloor
\]
and combining all the inequalities we'll get:
\[
\lim_{x\to d^{+}} f_{n}(x)=\lim_{x\to d^{+}}\lfloor x\cdot L_d^{-} + \underbrace{kx-\lfloor kx \rfloor}_{\text{$\geq 0$}} \rfloor + \lfloor kx \rfloor\geq\lim_{x\to d^{-}}\lfloor x\cdot L_d^{-}\rfloor + \lfloor kd \rfloor > \lim_{x\to d^{-}} f_{n}(x)
\]
\end{proof}
\begin{theorem}
Given the interval $[k,k+1)$ where $k\in\mathbb{Z^{+}}$, then:
\[
\lim_{n\to+\infty} |P(k,k+1,f_n)| = \infty
\]
Where $|P(k,k+1,f_n)|$ represents the cardinality of the set of all the discontinuity points of $f_n$ in the interval $[k,k+1)$.
\end{theorem}
\begin{proof}
Since from Theorem \ref{t5} we know that:
\[
P(f_1)\subset P(f_2)\subset\dots\subset P(f_n) \implies P(k,k+1,f_1)\subset P(k,k+1,f_2)\subset\dots\subset P(k,k+1,f_n)
\]
it's sufficient to show that $\forall n\geq 2, \exists d\in P(k,k+1,f_n)\setminus P(k,k+1,f_{n-1})$. In fact we'll show that:
\[
d = k+\frac{1}{k^{n-1}} \in P(k,k+1,f_n)\setminus P(k,k+1,f_{n-1})
\]
First we want to prove that:
\[
k+\frac{1}{k^{n}} \in P(k,k+1,f_{n+1})
\]
Or, using the definition of this set:
\[
\lim_{x\to (k+\frac{1}{k^{n}})^{-}} f_{n+1}(x) \not = \lim_{x\to (k+\frac{1}{k^{n}})^{+}} f_{n+1}(x)
\]
In fact we'll prove by induction on $n\geq 2$ that using $\delta=\frac{1}{k^{n}}$:
\[
f_{n+1}(x)=k^{n+1} \mbox{ $\forall x\in\Bigl(k+\frac{1}{k^{n}}-\delta,k+\frac{1}{k^{n}}\Bigr)$} \mbox{ while } f_{n+1}(x)=k^{n+1}+1 \mbox{ $\forall x\in\Bigl(k+\frac{1}{k^{n}},k+\frac{1}{k^{n}}+\delta\Bigr)$}
\]
Assuming that true we will have:
\[
\lim_{x\to (k+\frac{1}{k^{n}})^{-}} f_{n+1}(x) = k^{n+1} \mbox{ while } \lim_{x\to (k+\frac{1}{k^{n}})^{+}} f_{n+1}(x) = k^{n+1}+1
\]
which is our thesis. \\
By the induction hypothesis:
\[
f_n(x) = k^{n} \mbox{ $\forall x \in\Bigl(k,k+\frac{1}{k^{n-1}}\Bigr)$}
\]
\[
f_n(x) = k^{n}+1 \mbox{ $\forall x \in\Bigl(k+\frac{1}{k^{n-1}},k+\frac{2}{k^{n-1}}\Bigr)$}
\]
But using the definition of $f_{n+1}$ with $\delta=\frac{1}{k^{n}}<\frac{1}{k^{n-1}}$ we'll have:
\[
f_{n+1}=\lfloor x\cdot f_n(x) \rfloor = \lfloor x\cdot k^{n} \rfloor \mbox{ $\forall x \in\Bigl(k,k+\frac{1}{k^{n}}\Bigr)\subset\Bigl(k,k+\frac{1}{k^{n-1}}\Bigr) $}
\]
\[
f_{n+1}=\lfloor x\cdot f_n(x) \rfloor = \lfloor x\cdot k^{n} \rfloor\mbox{ $\forall x \in\Bigl(k+\frac{1}{k^{n}},k+\frac{2}{k^{n}}\Bigr)\subset\Bigl(k,k+\frac{1}{k^{n-1}}\Bigr)$}
\]
But furthermore:
\[
k<x<k+\frac{1}{k^{n}} \implies k^{n+1}<k^{n}x<k^{n+1}+1
\]
\[
k+\frac{1}{k^{n}}<x<k+\frac{2}{k^{n}} \implies k^{n+1}+1<k^{n}x<k^{n+1}+2
\]
And finally:
\[
f_{n+1}=\lfloor x\cdot k^{n} \rfloor = k^{n+1} \mbox{ $\forall x \in\Bigl(k,k+\frac{1}{k^{n}}\Bigr)$}
\]
\[
f_{n+1}=\lfloor x\cdot k^{n} \rfloor = k^{n+1}+1 \mbox{ $\forall x \in\Bigl(k+\frac{1}{k^{n}},k+\frac{2}{k^{n}}\Bigr)$}
\]
Now we know that $k+\frac{1}{k^{n}}\in P(k,k+1,f_{n+1})$, and we want to prove that $k+\frac{1}{k^{n}}\not\in P(k,k+1,f_{n})$.
This is true in fact:
\[
\lim_{x\to(k+\frac{1}{k^{n}})^{-}}f_n(x) = \lim_{x\to(k+\frac{1}{k^{n}})^{+}}f_n(x) = k^{n}
\]
In order to prove it, as seen before, we know that:
\[
f_n(x) = k^{n} \mbox{ $\forall x \in\Bigl(k,k+\frac{1}{k^{n-1}}\Bigr)$}
\]
But then:
\[
f_n(x) = k^{n} \mbox{ $\forall x \in\Bigl(k,k+\frac{1}{k^{n}}\Bigr)\subset\Bigl(k,k+\frac{1}{k^{n-1}}\Bigr) $}
\]
\[
f_n(x) = k^{n} \mbox{ $\forall x \in\Bigl(k+\frac{1}{k^{n}},k+\frac{2}{k^{n}}\Bigr)\subset\Bigl(k,k+\frac{1}{k^{n-1}}\Bigr)$}
\]
$\forall k\geq 2$, $\forall n\geq 1$. So we'll have:
\[
\underbrace{\{k+\frac{1}{k},\dots\}}_{\text{$P(k,k+1,f_2)$}}\subset\underbrace{\{k+\frac{1}{k},k+\frac{1}{k^{2}},\dots\}}_{\text{$P(k,k+1,f_3)$}}\subset\dots\subset\underbrace{\{k+\frac{1}{k},k+\frac{1}{k^{2}},\dots,k+\frac{1}{k^{n}},\dots\}}_{\text{$P(k,k+1,f_{n+1})$}}
\]
Where $P(k,k+1,f_{n+1})$ has $n$ terms of the form $k+\frac{1}{k^{r}}$, where $r\in\{1,\dots,n\}$. 
\end{proof}
\subsection{Script in Mathematica language}
In order to compute discontinuity points of the function $f_n$ in a given interval, is possible to use this script: \cite{2}
\begin{lstlisting}[language=Mathematica,caption={To compute discontinuity points of $f_3$ in $I=(2,3)$}]
FunctionDomain[{D[Floor[x*Floor[x*Floor[x]]], x], 2 < x < 3}, x]
  \end{lstlisting} 
  which gives:
  \begin{lstlisting}
  (*2 < x < 9/4 || 9/4 < x < 5/2 || 5/2 < x < 13/5 || 13/5 < x < 14/5 ||14/5 < x < 3*)
  \end{lstlisting}
  where $\frac{9}{4},\frac{5}{2},\frac{13}{5},\frac{14}{5}$ are in fact all the discontinuity points in that interval.
\begin{conjecture} \label{conj7}
Let $f_3(x)$ be defined as before, $P(a,b,f_3)$ the set of discontinuity points of $f_3(x)$ in the open interval $[a,b)$ and $P(f_3)$ the set of all discontinuity points over the domain $x\geq 1$. Then:
\[
P(f_3) = P(f_1)\cup P(f_2)\cup\bigcup_{k=1}^{+\infty}\bigcup_{i=0}^{k-1}\Bigl\{k+\frac{(k+1)i}{k^{2}+i},k+\frac{(k+1)i+1}{k^{2}+i},\dots,k+\frac{(k+1)i+k-1}{k^{2}+i}\Bigr\}
\]
Or:
\[
P(f_3) = P(f_1)\cup P(f_2)\cup\bigcup_{k=1}^{+\infty}\bigcup_{i=0}^{k-1}\bigcup_{p=0}^{k-1}\Bigl\{k+\frac{(k+1)i+p}{k^{2}+i}\Bigr\}
\]
where:
\[
|J(k,f_3)| = \frac{k^{3}-1}{k-1}=k^{2}+k+1 \mbox{ $\forall k\in\mathbb{Z^{+}}$}
\]
and:
\[
\Bigl|J\Bigl(k+\frac{r}{k}, f_3\Bigr)\Bigr| = k+1 \mbox{ $\forall k\in\mathbb{Z^{+}}$ where $r\in\{1,2,\dots,k-1\}$}
\]
and:
\[
\Bigl|J\Bigl(p, f_3\Bigr)\Bigr| = 1 \mbox{ $\forall p\in P(f_3)\setminus P(f_2)\cap P(f_1)$}
\]
\end{conjecture}
For example, consider the interval $[3,4)$. We've already said that:
\[
P(3,4,f_3)=\Bigl\{3,3+\frac{1}{9},3+\frac{2}{9},3+\frac{1}{3},3+\frac{2}{5},3+\frac{1}{2},3+\frac{3}{5},3+\frac{2}{3},3+\frac{8}{11},3+\frac{9}{11},3+\frac{10}{11}\Bigr\}
\]
This should represent this set:
\[
P(3,4,f_1)\cup P(3,4,f_2)\cup\bigcup_{i=0}^{2}\bigcup_{p=0}^{2}\Bigl\{3+\frac{4i+p}{9+i}\Bigr\} \mbox{ ($k=3$ in the formula above)}
\]
In fact:
\[
P(3,4,f_3)=\Bigl\{3\Bigr\}\cup\Bigl\{3,3+\frac{1}{3},3+\frac{2}{3}\Bigr\}\cup\Bigl\{3,3+\frac{1}{9},3+\frac{2}{9},3+\frac{4}{10},3+\frac{5}{10},3+\frac{6}{10},3+\frac{8}{11},3+\frac{9}{11},3+\frac{10}{11}\Bigr\}
\]
\subsection{Considerations of conjecture \ref{conj7}}
As in Theorem \ref{t4} we would like to prove our result using double inclusion of sets. But from Theorem \ref{t5} we know that:
\[
P(f_1)\subset P(f_2)\subset P(f_3)
\]
So it's to sufficient to prove that all the elements in the set $P(f_3)\setminus P(f_1)\cup P(f_2)$ are jump discontinuities for $f_3$ (because we already known that the elements in $P(f_1)\cup P(f_2)$ are jump discontinuities for $f_3$).
In general, the set: $P(f_3)\setminus P(f_1)\cup P(f_2)$ is given by:
\[
P(f_3)\setminus P(f_1)\cup P(f_2) = \bigcup_{k=1}^{+\infty}\bigcup_{i=0}^{k-1}\bigcup_{p=0}^{k-1}\Bigl\{k+\frac{(k+1)i+p}{k^{2}+i}\Bigr\}\setminus\{k\}
\]
We would like to show that:
\[
\lim_{x\to (k+\frac{(k+1)i+p}{k^{2}+i})^{-}}f_3(x) = k^{3}+2ik+i+p-1
\]
While:
\[
\lim_{x\to (k+\frac{(k+1)i+p}{k^{2}+i})^{+}}f_3(x) = k^{3}+2ik+i+p 
\]
For example:
\[
\lim_{x\to(3+\frac{4\cdot 1+0}{9+1})^{-}}f_3(x) = \lim_{x\to \frac{34}{10}^{-}}f_3(x)= 3^{3}+2\cdot1\cdot3+0-1=33
\]
\section{Other generalizations}
Given the result obtained before, it's trivial to prove that under the same conditions: \\
$\forall m\in\mathbb{Z^{+}}$
\[
\lim_{x\to k^{-}}\underbrace{\Biggl\lfloor x^{m}\Bigl\lfloor x^{m} \lfloor\dots\rfloor\Bigr\rfloor\Biggr\rfloor}_{\text{$n$ times}}=\frac{1}{k^{m}-1}\Bigl[(k^{m}-2)\cdot k^{m\cdot n}+1\Bigr]
\]

For example, for $k=4, n=3, m=2$ we have that:
\[
\lim_{x\to 4^{-}}\Biggl\lfloor x^{2}\Bigl\lfloor x^{2} \lfloor x^{2}\rfloor\Bigr\rfloor\Biggr\rfloor = 3823 = \frac{1}{4^{2}-1}\Bigl[(4^{2}-2)\cdot 4^{2\cdot 3}+1\Bigr] = \frac{57345}{15}
\]

\end{document}